\documentclass[11pt]{article}
\usepackage{amsmath,amsthm,amssymb,amsfonts}
\usepackage[width=16.5cm,height=23cm]{geometry}
\usepackage{enumerate}
\usepackage{mathrsfs}
\usepackage[cmtip,all]{xy}

\usepackage{tikz}

\theoremstyle{plain}
\newtheorem{theorem}{Theorem}[section]
\newtheorem*{theorem*}{Theorem}
\newtheorem{lemma}[theorem]{Lemma}
\newtheorem{proposition}[theorem]{Proposition}
\newtheorem*{proposition*}{Proposition}

\newtheorem*{corollary*}{Corollary}

\newtheorem{remark*}{Remark}

\newtheorem{defn*}{Definition}

\def\colim{\operatornamewithlimits{colim}}
\def\hocolim{\operatornamewithlimits{hocolim}}

\def\geom{\operatorname{geom}}
\def\Gr{\mathbf{Gr}}
\def\HGr{\mathbf{HGr}}
\def\ZZ{\mathbb{Z}}
\def\HP{\mathbf{HP}}
\def\ZZ{\mathbf{Z}}
\def\AA{\mathbf{A}}
\def\GG{\mathbf{G}}
\def\PP{\mathbf{P}}
\def\Th{\mathbf{Th}}

\newcommand{\KT}{\mathbf{KW}}
\newcommand{\KQ}{\mathbf{KQ}}
\newcommand{\BO}{\mathbf{BO}}
\newcommand{\s}{\mathsf{s}}
\newcommand{\MZZ}{\mathbf{MZ}}

\begin{document}

\title{Cellularity of hermitian $K$-theory and Witt-theory}
\author{Oliver R{\"o}ndigs, Markus Spitzweck, Paul Arne {\O}stv{\ae}r}
\date{March 16, 2016}
\maketitle

\begin{abstract}
Hermitian $K$-theory and Witt-theory are cellular in the sense of stable motivic homotopy theory over any base scheme without points of characteristic two.
\end{abstract}

\section{Introduction}

The notion of a cellular object in motivic homotopy theory is intrinsically linked to the geometry of motivic spheres $S^{p,q}$ \cite{MR2153114}.
Suppose the smooth scheme $X$ admits a filtration by closed subschemes
\begin{equation*}
\emptyset
\subset
X_{0}
\subset
\dotsm
\subset
X_{n-1}
\subset
X_{n}
=
X,
\end{equation*}
where $X_{i}\smallsetminus X_{i-1}$ is a disjoint union of affine spaces $\AA^{n_{ij}}$.
Examples of such filtrations arise in the context of Bia{\l}ynicki-Birula decompositions for $\GG_{m}$-action on smooth projective varieties \cite{MR0366940}, 
cf.~\cite{MR2178658} for a more recent implementation.
By homotopy purity \cite[Theorem 3.2.23]{MR1813224} for Thom spaces of normal bundles of closed embeddings, 
there is a homotopy cofiber sequence 
\begin{equation*}
\xymatrix{
X\smallsetminus X_{i}
\ar[r] & 
X\smallsetminus X_{i-1}
\ar[r] & 
\Th(\mathcal{N}_{i}).
}
\end{equation*}
By assumption the normal bundle $\mathcal{N}_{i}$ is trivial. 
Thus the splitting $\Th(\mathcal{N}_{i})\cong\bigvee_{j}S^{2n_{ij},n_{ij}}$ and the two-out-of-three property for stably cellular objects \cite[Lemma 2.5]{MR2153114} 
imply inductively that $X$ is stably cellular in the sense of \cite[Definition 2.10]{MR2153114}.

In this paper we employ a similar strategy to prove cellularity for Thom spaces of direct sums of tautological sympletic bundles over quaternionic Grassmannians. 
This allows us to show cellularity of the motivic spectra representing hermitian $K$-theory and Witt-theory \cite{MR2122220}.
By a base scheme we mean any regular noetherian separated scheme of finite Krull dimension.
\begin{theorem}
\label{theorem:KQcellular}
Suppose all points on the base scheme have residue characteristic unequal to two. 
Then hermitian $K$-theory $\KQ$ and Witt-theory $\KT$ are cellular motivic spectra.
\end{theorem}

For a related antecedent result showing cellularity of algebraic $K$-theory, 
see \cite[Theorem 6.2]{MR2153114}.
The proof of Theorem \ref{theorem:KQcellular} exploits the geometry of quaternionic Grassmannians and the explicit model for hermitian $K$-theory from \cite{panin-walterKO}.

Recent applications of $\KQ$ and $\KT$ concern computations of stable homotopy groups of motivic spheres \cite{MR2061856}, \cite{MR3255457}, \cite{roendigs-spitzweck-oestvaer.1line}, 
and a proof of the Milnor conjecture on quadratic forms \cite{roendigs-oestvaer.hermitian}.
For cellular motivic spectra one has the powerful fact that stable motivic weak equivalences are detected by $\pi_{\ast,\ast}$-isomorphisms \cite[Corollary 7.2]{MR2153114}.
Our main motivation for proving Theorem \ref{theorem:KQcellular} is that it is being used in the computation of the slices of $\KQ$ in \cite[Theorem 2.14]{roendigs-spitzweck-oestvaer.1line}. 
In terms of motivic cohomology with integral and mod-$2$ coefficients, 
the result is
\begin{equation*}
\label{equation:hermitianktheoryslices}
\s_{q}(\KQ)
\cong
\begin{cases}
\Sigma^{2q,q}\MZZ
\vee  
\bigvee_{i<\frac{q}{2}}\Sigma^{2i+q,q}\MZZ/2 & \text{$q$ even}  \\
\bigvee_{i<\frac{q+1}{2}}\Sigma^{2i+q,q}\MZZ/2 & \text{$q$ odd.}  \\
\end{cases}
\end{equation*}
In turn, 
this is an essential ingredient in our proof of Morel's $\pi_{1}$-conjecture in \cite{roendigs-spitzweck-oestvaer.1line}. 
It is an interesting problem to make sense of Theorem \ref{theorem:KQcellular} without any assumptions on the points of the base scheme.

This short paper is organized into Section \ref{section:cellularobjects} on basic properties of motivic cellular spectra, 
Section \ref{section:quaternionicGrassmannians} on the geometry of quaternionic Grassmannians, 
and Section \ref{section:hermitianKtheoryandWitttheory} on hermitian $K$-theory and Witt-theory.

\section{Cellular objects}
\label{section:cellularobjects}

The subcategory of cellular spectra in the motivic stable homotopy category is the smallest full localizing subcategory that contains all suspensions of the sphere spectrum, 
cf.~\cite[\S2.8]{MR2153114}.
For our purposes it suffices to know four basic facts about cellular motivic spectra.
First we recall part (3) of Definition 2.1 in \cite{MR2153114}.
\begin{lemma}
\label{lemma:colimitcellular}
The homotopy colimit of a diagram of cellular motivic spectra is cellular.
\end{lemma}
The second fact is a specialization of \cite[Lemma 2.4]{MR2153114}.
\begin{lemma}
\label{lemma:suspensioncellular}
Let $E$ be a motivic spectrum and let $p$, $q$ be integers.
Then $E$ is cellular if and only if its $(p,q)$-suspension $\Sigma^{p,q}E$ is cellular.
\end{lemma}
The third fact is a specialization of \cite[Lemma 2.5]{MR2153114}.
\begin{lemma}
\label{lemma:homotopycofibersequence}
If $E\to F\to G$ is a homotopy cofiber sequence of motivic spectra such that any two of $E$, $F$, and $G$ are cellular, 
then so is the third.
\end{lemma}
Finally, we recall Lemma 3.2 in \cite{MR2153114}.
\begin{lemma}
\label{lemma:disjointcellular}
If $E_{i}$ is a cellular motivic spectrum for all $i\in I$, 
then $\coprod_{i\in I} E_{i}$ is cellular.
\end{lemma}

\section{Quaternionic Grassmannians}
\label{section:quaternionicGrassmannians}

The quaternionic Grassmannian $\HGr(r,n)$ is the open subscheme of the ordinary Grassmannian $\Gr(2r,2n)$ parametrizing $2r$-dimensional subspaces of the trivial vector bundle 
$\mathcal{O}^{\oplus 2n}$ on which the standard symplectic form is nondegenerate.
It is smooth affine of dimension $4r(n-r)$ over the base scheme.
Let $\mathcal{U}_{r,n}$ be short for the tautological symplectic subbundle of rank $2r$ on $\HGr(r,n)$.
It is the restriction to $\HGr(r,n)$ of the tautological subbundle of $\Gr(2r,2n)$ together with the restriction to $\mathcal{U}_{r,n}$ of the standard symplectic form on 
$\mathcal{O}^{\oplus 2n}$.

More generally, 
to every symplectic bundle $(\mathcal{E},\phi)$ one associates the quaternionic Grassmannian $\HGr(r,\mathcal{E},\phi)$;
it is the open subscheme of the Grassmannian $\Gr(2r,\mathcal{E})$ parametrizing $2r$-dimensional subspaces of the fibers of $\mathcal{E}$ on which $\phi$ is nondegenerate.
Associated to the trivial rank $2n-2$ symplectic bundle $(\mathcal{E},\psi)$ is the bundle $\mathcal{F}=\mathcal{O}\oplus\mathcal{E}\oplus\mathcal{O}$ equipped with the 
direct sum of $\psi$ and the hyperbolic symplectic form, 
i.e., 
\[
\begin{bmatrix}
0   &  0     &  1 \\
0   &  \psi  &  0 \\
-1  &  0     &  0
\end{bmatrix}.
\]
For simplicity we write $\HGr(\mathcal{E})$ for $\HGr(r,\mathcal{E},\psi)$ and likewise for $\mathcal{F}$.

The normal bundle $N$ of the embedding $\HGr(\mathcal{E})\subset\HGr(\mathcal{F})$ is the tensor product $\mathcal{U}_{\mathcal{E}}^{\vee}\otimes\mathcal{O}^{\oplus 2}$ for the dual 
of the tautological symplectic subbundle of rank $2r$ on $\HGr(\mathcal{E})$.
Theorem 4.1 in \cite{panin-walterGP} shows that $N$ is naturally isomorphic to an open subscheme of $\Gr(2r,\mathcal{F})$ and there is a decomposition $N=N^{+}\oplus N^{-}$;
here, 
$N^{+}=\HGr(\mathcal{F})\cap\Gr(2r,\mathcal{O}\oplus\mathcal{E})$ and $N^{-}=\HGr(\mathcal{F})\cap\Gr(2r,\mathcal{E}\oplus\mathcal{O})$ have intersection $\HGr(\mathcal{E})$.
Thus there are natural vector bundle isomorphisms $N^{+}\cong N^{-}\cong\mathcal{U}_{r,n-1}$ and the normal bundle $\mathcal{N}$ of $N^{+}$ in $\HGr(\mathcal{F})$ is isomorphic to 
$\pi_{+}^{\ast}\mathcal{U}_{r,n-1}$ for the bundle projection $\pi_{+}\colon N^{+}\to\HGr(\mathcal{E})$.
Moreover, 
there is a vector bundle isomorphism between the restriction $\mathcal{U}_{r,n}\vert N^{+}$ of $\mathcal{U}_{r,n}$ to $N^{+}$ and $\pi_{+}^{\ast}\mathcal{U}_{r,n-1}$.
For $r\leq n-1$, 
let $Y$ denote the complement of $N^{+}$ in $\HGr(\mathcal{F})$ \cite[(5.1)]{panin-walterGP}.

\begin{proposition}
\label{proposition:finitecellular}
For $m\geq 0$ the suspension spectrum of the Thom space of the vector bundle $\mathcal{U}_{r,n}^{\oplus m}$ on $\HGr(r,n)$ is a finite cellular spectrum.
In particular, 
$\Sigma^{\infty}\HGr(r,n)_{+}$ is a cellular spectrum.
\end{proposition}
\begin{proof}
The proof proceeds by a double induction argument on $r$ and $n\geq r$.
The base cases $\HGr(0,n)$ and $\HGr(n,n)$ are clear, so we may assume $0<r<n$.
Define the motivic space $Z$ by the homotopy cofiber sequence
\begin{equation}
\label{equation:homotopycofibersequence}
\xymatrix{
\Th(\mathcal{U}_{r,n}^{\oplus m}\vert Y)
\ar[r] & 
\Th(\mathcal{U}_{r,n}^{\oplus m})
\ar[r] & 
Z. }
\end{equation}
According to \cite[Lemma 3.5]{MR2771593} there is a canonical isomorphism in the motivic homotopy category
\[
Z
\cong
\Th(\mathcal{U}_{r,n}^{\oplus m}\vert N^{+}\oplus \mathcal{N}).
\]
Using the above we note $\mathcal{U}_{r,n}^{\oplus m}\vert N^{+}\oplus \mathcal{N}\cong \pi_{+}^{\ast}\mathcal{U}_{r,n-1}^{\oplus (m+1)}$ and hence there are canonical isomorphisms 
\[
Z
\cong
\Th(\pi_{+}^{\ast}\mathcal{U}_{r,n-1}^{\oplus (m+1)})
\cong
\Th(\mathcal{U}_{r,n-1}^{\oplus (m+1)}).
\]
By induction hypothesis $\Sigma^{\infty}Z$ is a finite cellular spectrum.
Thus Lemma \ref{lemma:homotopycofibersequence} and \eqref{equation:homotopycofibersequence} reduce the proof to showing that $\Sigma^{\infty}\Th(\mathcal{U}_{r,n}^{\oplus m}\vert Y)$ 
is a finite cellular spectrum.
To this end we recall parts of Theorem 5.1 in \cite{panin-walterGP}:
There exists maps 
\begin{equation*}
\xymatrix{
Y & Y_{1} \ar[l]_-{g_{1}} & Y_{2} \ar[l]_-{g_{2}} \ar[r]^-{q} & \HGr(r-1,\mathcal{E},\psi), 
}
\end{equation*}
where $g_{i}$ and $q$ are Zariski locally trivial torsors over vector bundles of rank $2r-i$ and $4n-3$,
respectively.
Moreover,
$g_{2}^{\ast}g_{1}^{\ast}\mathcal{U}_{r,n}$ is isomorphic to $\mathcal{O}_{Y_{2}}^{2}\oplus q^{\ast}\mathcal{U}_{r-1,n}$.
Invoking \cite[\S3.2, Example 2.3]{MR1813224} this implies the canonical isomorphisms 
\[
\Sigma^{\infty}\Th(\mathcal{U}_{r,n}^{\oplus m}\vert Y)
\cong
\Sigma^{\infty}\Th(g_{2}^{\ast}g_{1}^{\ast}\mathcal{U}_{r,n}^{\oplus m}\vert Y)
\cong
\Sigma^{\infty}\Th(\mathcal{O}_{Y_{2}}^{2m}\oplus q^{\ast}\mathcal{U}_{r-1,n}^{\oplus m})
\cong
\Sigma^{2m,m}\Sigma^{\infty}\Th(\mathcal{U}_{r-1,n}^{\oplus m}).
\]
Here, 
the suspension spectrum of $\Th(\mathcal{U}_{r-1,n}^{\oplus m})$ is finite cellular by the induction hypothesis.
This finishes the proof using Lemma \ref{lemma:suspensioncellular}.
\end{proof}

\section{Hermitian $K$-theory and Witt-theory}
\label{section:hermitianKtheoryandWitttheory}

In this section we finish the proof of Theorem \ref{theorem:KQcellular} stated in the introduction.

The quaternionic plane $\HP^{1}$ is the first quaternionic Grassmannian $\HGr(1,2)$.
In the pointed unstable motivic homotopy category, 
$(\HP^{1},x_{0})$ is isomorphic to the two-fold smash product of the Tate object $T\equiv\AA^{1}/\AA^{1}\smallsetminus\{0\}$. 
It follows that the $\AA^{1}$-mapping cone $\HP^{1+}$ of the rational point $x_{0}\colon S\to\HP^{1}$ is isomorphic to $T^{\wedge 2}$.
Hence the stable homotopy category of $\HP^{1+}$-spectra is equivalent to the standard model for the stable motivic homotopy category \cite[Theorem 12.1]{panin-walterKO}.

Theorem 12.3 in \cite{panin-walterKO} shows there is an isomorphism between hermitian $K$-theory $\KQ$ and an $\HP^{1+}$-spectrum $\BO^{\geom}$.
For $n$ odd, 
$\BO^{\geom}_{2n}=\ZZ\times\HGr$ \cite[(12.5)]{panin-walterKO}.
Here $\HGr$ denotes the infinite quaternionic Grassmannian, 
i.e., 
the sequential colimit 
\[
\underset{n}{\colim}\;\HGr(n,2n).
\]
We note that the transition maps in the colimit are defined in \cite[(8.1)]{panin-walterKO}.
The motivic space $\ZZ\times\HGr$ is pointed by $(0,\HGr(0,0))$.
Thus $\KQ$ is isomorphic to the homotopy colimit 
\begin{equation}
\label{equation:KQcolimit}
\underset{n\text{ odd}}{\hocolim}\;\Sigma^{4n,2n}\Sigma^{\infty}\ZZ\times\HGr.
\end{equation}
It remains to show cellularity of \eqref{equation:KQcolimit}.
Note that $\Sigma^{\infty}\ZZ\times\HGr$ is a homotopy colimit of cellular spectra by Lemma \ref{lemma:disjointcellular} and Proposition \ref{proposition:finitecellular}.
It follows that $\Sigma^{4n,2n}\Sigma^{\infty}\ZZ\times\HGr$ is cellular according to Lemmas \ref{lemma:colimitcellular} and \ref{lemma:suspensioncellular}.
We conclude the proof for $\KQ$ by applying Lemma \ref{lemma:colimitcellular}.

Cellularity of $\KT$ follows from that of $\KQ$ via Lemma \ref{lemma:colimitcellular} and the description of $\KT$ as the homotopy colimit of the diagram 
\begin{equation*}
\xymatrix{
\KQ \ar[r]^-{\eta} & \Sigma^{-1,-1}\KQ \ar[rr]^-{\Sigma^{-1,-1}\eta} && \Sigma^{-2,-2}\KQ \ar[rr]^-{\Sigma^{-2,-2}\eta} && \dotsm
}
\end{equation*}
given in \cite[Theorem 6.5]{ananyevskiy.sl}. Here, 
$\eta$ is the first stable Hopf map induced by the canonical map $\AA^{2}\smallsetminus\{0\}\rightarrow\PP^{1}$.

\begin{footnotesize}

\end{footnotesize}
\vspace{0.1in}

\begin{center}
Institut f\"ur Mathematik, Universit\"at Osnabr\"uck, Germany.\\
e-mail: oliver.roendigs@uni-osnabrueck.de
\end{center}
\begin{center}
Institut f\"ur Mathematik, Universit\"at Osnabr\"uck, Germany.\\
e-mail: markus.spitzweck@uni-osnabrueck.de
\end{center}
\begin{center}
Department of Mathematics, University of Oslo, Norway.\\
e-mail: paularne@math.uio.no
\end{center}

\end{document}